\documentclass[12pt]{amsart}

\usepackage{amsmath, amssymb, amscd, amsthm, amsfonts, amsxtra}
\usepackage{graphicx, float}
\usepackage[T1]{fontenc}
\usepackage[utf8]{inputenc}
\usepackage{CJKutf8}
\usepackage{kotex}
\usepackage[all]{xy}
\usepackage{enumerate}
\usepackage{tikz-cd}
\usepackage{mathdots}
\usepackage{hyperref}

\usepackage{xypic}
\headsep=1cm \numberwithin{equation}{section}
\def\P{{\mathbb P}}

\def\Z{{\mathbb Z}}
\newtheorem{theorem}{Theorem}[section]
\newtheorem{lemma}[theorem]{Lemma}

\newtheorem{proposition}[theorem]{Proposition}

\theoremstyle{definition}

\newtheorem{convention and reminder}[theorem]{Convention and Reminder}
\newtheorem{convention and remark}[theorem]{Convention and Remark}
\newtheorem{definition and remark}[theorem]{Definition and Remark}

\newtheorem{reminders and definition}[theorem]{Reminders and Definition}

\newtheorem{notation and remarks}[theorem]{Notation and Remarks}
\newtheorem{notation and remark}[theorem]{Notation and Remark}
\newtheorem{example}[theorem]{Example}

\newcommand\Ker{\operatorname{\Ker}}

\newcommand\Supp{\operatorname{Supp}}

\def\QR{\mathsf{QR}}

\swapnumbers

\title[Rank 3 quadratic generators: the characteristic 3 case]{Rank 3 quadratic generators of Veronese embeddings: the characteristic 3 case}

\author{Donghyeop Lee}
\email{donghyeoplee@yonsei.ac.kr}
\address{Donghyeop Lee, Department of Mathematics, Yonsei University, 50 Yonsei-ro, Seodaemun-gu, Seoul 03722, Republic of Korea}

\author{Euisung Park}
\email{euisungpark@korea.ac.kr}
\address{Euisung Park, Department of Mathematics, Korea University, Seoul 136-701, Republic of Korea}

\author{Saerom Sim}
\email{saeromiii@korea.ac.kr}
\address{Saerom Sim, Department of Mathematics, Korea University, Seoul 136-701, Republic of Korea}

\begin{document}

%\date{\today}

\begin{abstract}
This paper investigates property QR(3) for Veronese embeddings over an algebraically closed field of characteristic $3$. We determine the rank index of $(\P^n , \mathcal{O}_{\P^n} (d))$ for all $n \geq 2$, $d \geq 3$, proving that it equals $3$ in these cases. Our approach adapts the inductive framework of \cite{HLMP}, re-proving key lemmas for characteristic $3$ to establish quadratic generation by rank $3$ forms. We further compute the codimension of the span of rank $3$ quadrics in the space of quadratic equations of the second Veronese embedding, showing it grows as ${n+1 \choose 4}$. This provides a clear explanation of the exceptional behavior exhibited by the second Veronese embedding in characteristic $3$. Additionally, we show that for a general complete intersection of quadrics $X \subset \P^r$ of dimension at least $3$, the rank index of $(X,\mathcal{O}_X (2))$ is $4$, thereby confirming the optimality of our main bound. These results complete the classification of the rank index for Veronese embeddings when ${\rm char}(\mathbb{K}) \ne 2$.
\end{abstract}

\maketitle \setcounter{page}{1}

\section{Introduction}
\noindent This paper continues the investigation initiated in \cite{HLMP} regarding the property $\QR (3)$ of the Veronese embeddings of arbitrary projective varieties.

Let $X$ be a projective variety over an algebraically closed field $\mathbb{K}$ with ${\rm char}(\mathbb{K}) \ne 2$, and let $L$ be a very ample line bundle on $X$ defining a linearly normal embedding
$$X \subset \P^r , \quad r = h^0 (X,L)-1.$$
For the past few decades, considerable attention has been devoted to establishing conditions under which the homogeneous ideal $I(X,L)$ of $X$ in $\P^r$ is generated by quadrics and its first few syzygy modules are generated by linear syzygies (cf. \cite{Gre1}, \cite{Gre2}, \cite{GL}, \cite{EL}, \cite{GP}, \cite{Ina}, \cite{LP}, etc). Another important direction in the study of defining equations of $X$ is to investigate whether $I(X,L)$ can be defined by $2$-minors of one or several linear determinantal presentations (cf. \cite{EKS}, \cite{Pu}, \cite{Ha}, \cite{B}, \cite{SS}, etc). In such a case, $I(X)$ admits a generating set consisting of quadrics of rank $\leq 4$. Along this line, we say that $(X,L)$ satisfies property $\QR (k)$ if the ideal $I(X,L)$ is generated by quadrics of rank at most $k$. When $I(X,L)$ is generated by quadrics, we define the rank index of $(X,L)$, denoted $\mbox{rank-index}(X,L)$, as the minimal integer $k$ such that $(X,L)$ satisfies property $\QR (k)$. As $X$ is irreducible and nondegenerate in $\P^r$, one always has
$$\mbox{rank-index}(X,L) \geq 3.$$
Recent results show that the lowest bound is attained in a wide range of cases; that is, there exist many examples where $\mbox{rank-index}(X,L)$ attains the minimal possible value $3$.

\begin{theorem}[Theorem 1.1 and Theorem 1.3 in \cite{HLMP}, Theorem 1.1 in \cite{P22}]\label{thm:HLMP}
Suppose that ${\rm char}(\mathbb{K}) \neq 2,3$. Then

\renewcommand{\descriptionlabel}[1]%
             {\hspace{\labelsep}\textrm{#1}}
\begin{description}
\setlength{\labelwidth}{13mm}
\setlength{\labelsep}{1.5mm}
\setlength{\itemindent}{0mm}

\item[$(1)$] The rank index of $(\P^n ,\mathcal{O}_{\P^n} (d))$ is $3$ for all $n \geq 1$ and $d \geq 2$.
\item[$(2)$] Let $L$ be a very ample line bundle on a projective variety $X$ defining a linearly normal embedding $X \subset \P^r$. If $m \geq 2$ is an integer such that $X$ in $\P^r$ is $j$-normal for all $j \geq m$ and $I(X,L)$ is generated by forms of degree $\leq m$, then
$$\mbox{rank-index}(X,L^d ) = 3 \quad \mbox{for all} \quad d \geq m.$$
\item[$(3)$] Let $\mathcal{C}$ be a projective integral curve of arithmetic genus $g$, and let $\mathcal{L}$ be a line bundle on $\mathcal{C}$ of degree $d \geq 4g+4$. Then $\mbox{rank-index} (\mathcal{C},\mathcal{L}) = 3$.
\end{description}
\end{theorem}

Theorem \ref{thm:HLMP} provides a strong evidence for Conjecture 6.1 in \cite{HLMP}, which predicts that \\

\begin{enumerate}
\item[$(\star)$] if $X$ is a projective variety, then $\mbox{rank-index}(X,L) = 3$ for all sufficiently ample line bundles $L$ on $X$.  \\
\end{enumerate}

\noindent In contrast, when ${\rm char}(\mathbb{K}) = 3$, the rank index of $(\P^n ,\mathcal{O}_{\P^n} (d))$ is known only in the cases where $n=1$ and $d \geq 2$, or $n \geq 2$ and $d=2$. More precisely, we have
\begin{equation*}
\mbox{rank-index}(\P^n ,\mathcal{O}_{\P^n} (d)) = \begin{cases}
3 & \quad \mbox{if} \quad n=1 \quad \mbox{and} \quad d \geq 2 \quad \mbox{or} \quad n=d=2; \quad \mbox{and} \\
4 & \quad \mbox{if} \quad n \geq 3 \quad \mbox{and} \quad d=2.
\end{cases}
\end{equation*}
For details, see \cite[Theorem 1.2]{HLMP}. Along these lines, the main purpose of this paper is to determine the rank index of $(\P^n ,\mathcal{O}_{\P^n} (d))$ in the remaining cases, namely when ${\rm char}(\mathbb{K}) = 3$, $n \geq 2$ and $d \geq 3$.

Our first main result is the following.

\begin{theorem}\label{thm:main theorem 1}
Suppose that ${\rm char}(\mathbb{K}) = 3$. Then

\renewcommand{\descriptionlabel}[1]%
             {\hspace{\labelsep}\textrm{#1}}
\begin{description}
\setlength{\labelwidth}{13mm}
\setlength{\labelsep}{1.5mm}
\setlength{\itemindent}{0mm}

\item[$(1)$] The rank index of $(\P^n ,\mathcal{O}_{\P^n} (d))$ is $3$ for all $n \geq 2$ and $d \geq 3$.
\item[$(2)$] Let $L$ be a very ample line bundle on a projective variety $X$ defining the linearly normal embedding $X \subset \P^r$. If $m \geq 3$ is an integer such that $X$ is $j$-normal for all $j \geq m$ and $I(X,L)$ is generated by forms of degree $\leq m$, then
$$\mbox{rank-index}(X,L^d ) = 3 \quad \mbox{for all} \quad d \geq m.$$
\end{description}
\end{theorem}

For the proof of Theorem \ref{thm:main theorem 1}.$(1)$, see Theorem \ref{boundlemma}. The proof of Theorem \ref{thm:main theorem 1}.$(2)$ can be found at the end of Section 3.

By Theorem \ref{thm:HLMP}.$(1)$ and Theorem \ref{thm:main theorem 1}.$(1)$, the problem of determining the rank index of $(\P^n,\mathcal{O}_{\P^n} (d))$ is solved for all
$n \geq 1$, $d \geq 2$ and ${\rm char}(\mathbb{K}) \neq 2$. In particular, the only exceptional cases occur when ${\rm char}(\mathbb{K}) = 3$, $n \geq 3$ and $d=2$; that is, for the pair $(\P^n , \mathcal{O}_{\P^n} (2))$.

Let us briefly explain the proof of Theorem \ref{thm:main theorem 1}.$(1)$. We begin by summarizing the structure of the proof of Theorem \ref{thm:HLMP}.$(1)$, which follows a  three-step strategy.\\

\begin{enumerate}
\item[\textbf{Step 1.}] When ${\rm char}(\mathbb{K}) \neq 2$, show that the rank index of $(\P^1 ,\mathcal{O}_{\P^1} (d))$ is $3$ for all $d \geq 2$.
\item[\textbf{Step 2.}] When ${\rm char}(\mathbb{K}) \neq 2, 3$, show that the rank index of $(\P^n ,\mathcal{O}_{\P^n} (2))$ is $3$ for all $n \geq 2$.
\item[\textbf{Step 3.}] When ${\rm char}(\mathbb{K}) \neq 2, 3$, show that the rank index of $(\P^n ,\mathcal{O}_{\P^n} (d))$ is $3$ for all $n \geq 1$ and $ d \geq 2$, using double induction on $n$ and $d$.\\
\end{enumerate}

In Step 2, the case ${\rm char}(\mathbb{K})=3$ is excluded since the rank index of $(\P^n ,\mathcal{O}_{\P^n} (2))=4$ for all $n \geq 3$ when ${\rm char}(\mathbb{K})=3$. Accordingly, the ranges of $n$ and $d$ in Step 2 and 3 must be adjusted to account for the case ${\rm char}(\mathbb{K})=3$. Also, there are some lemmas in \cite{HLMP} that play an important role in Step 3, some of which assume that ${\rm char}(\mathbb{K}) \neq 3$.

The proof of Theorem \ref{thm:main theorem 1}.$(1)$ follows a similar framework, employing double induction on $n \geq 1$ and $d \geq 3$. In our proof of Theorem \ref{thm:main theorem 1}.$(1)$, instead of the above Step 2 and Step 3 above, we show the following two statements.\\

\begin{enumerate}
\item[\textbf{Step 2*.}] When ${\rm char}(\mathbb{K}) = 3$, prove that the rank index of $(\P^n ,\mathcal{O}_{\P^n} (3))$ is $3$ for all $n \geq 1$.
\item[\textbf{Step 3*.}] When ${\rm char}(\mathbb{K}) = 3$, prove that the rank index of $(\P^n ,\mathcal{O}_{\P^n} (d))$ is $3$ for all $n \geq 1$ and $ d \geq 3$, using double induction on $n$ and $d$. \\
\end{enumerate}

Regarding \textbf{Step 3*}, certain lemmas in \cite{HLMP} required the assumption $\mbox{char}(\mathbb{K}) \neq 2,3$. To make the double induction argument work in characteristic $3$, we re-proved these lemmas specifically for $\mbox{char}(\mathbb{K}) = 3$, thereby completing the inductive step and ensuring the validity of the overall strategy.

The proof of Theorem \ref{thm:main theorem 1}.$(2)$ is obtained by combining Theorem \ref{thm:main theorem 1}.$(1)$ with the proof strategy of Theorem \ref{thm:HLMP}.$(2)$. \\

Next, our second main theorem revisits and strengthens \cite[Theorem 1.2]{HLMP}. To state the result precisely, let
$$X := \nu_2 (\P^n ) \subset \P^r$$
be the second Veronese embedding of $\P^n$ for some $n \geq 3$. Let $\Phi_3 (X)$ denote the set of all rank $3$ quadrics in the projective space $\P (I(X)_2 )$. It was shown in \cite{P24} that $\Phi_3 (X)$ is an irreducible projective variety of dimension $2n-2$. Let $\langle \Phi_3 (X) \rangle$ denote the span of $\Phi_3 (X)$ in $\P (I(X)_2 )$. Then the rank index of $(\P^n , \mathcal{O}_{\P^n} (2))$ is $3$ if and only if $\langle \Phi_3 (X) \rangle = \P (I(X)_2 )$. From this point of view, Theorem 1.1 and 1.2 in \cite{HLMP} show that the codimension of $\langle \Phi_3 (X) \rangle$ in $\P (I(X)_2 )$ is positive if and only if ${\rm char}(\mathbb{K}) = 3$ and $n \geq 3$.

Along this line, our second main result is the following.

\begin{theorem}\label{thm:main theorem 2}
Suppose that ${\rm char}(\mathbb{K}) = 3$, and let $X$ and $\langle \Phi_3 (X) \rangle$ be as above. Then
$$\mbox{codim} (  \langle \Phi_3 (X) \rangle ,\P ( I(X)_2 )) = {n+1 \choose 4} .$$
\end{theorem}

The proof of this result is given in the first half of Section 5.

This result clearly illustrates how the pair $(\P^n ,\mathcal{O}_{\P^n} (2))$ becomes increasingly far from satisfying property $\QR (3)$ as $n$ grows.\\

Finally, we consider the case where $(X,L)$ satisfies property $N_1$; that is, the linearly normal embedding of $X$ by $L$ is projectively normal and the ideal $I(X,L)$ is generated by quadratic polynomials. In this setting, one has $\mbox{rank-index}(X,L^d) = 3$ if either
\begin{enumerate}
\item[$(i)$] ${\rm char}(\mathbb{K}) \ne 2,3$ and $d \geq 2$ (cf. Theorem \ref{thm:HLMP}.$(2)$ or \cite[Corollary 5.2]{HLMP})
\end{enumerate}
or else
\begin{enumerate}
\item[$(ii)$] ${\rm char}(\mathbb{K}) = 3$ and $d \geq 3$ (cf. Theorem \ref{thm:main theorem 1}.$(2)$).
\end{enumerate}
But the value of $\mbox{rank-index}(X,L^2)$ remains unknown when ${\rm char}(\mathbb{K}) = 3$. In particular, the rank index of $(\P^n ,\mathcal{O}_{\P^n} (2)) = 4$ for all $n \geq 3$ if ${\rm char}(\mathbb{K}) = 3$, suggesting that, unlike the case ${\rm char}(\mathbb{K}) \ne 2,3$, the rank index of $(X,L^2)$ in ${\rm char}(\mathbb{K}) = 3$ may vary depending on $X$ and $L$. From this perspective, our third main theorem demonstrates that the condition $m \geq 3$ in Theorem \ref{thm:main theorem 1}.$(2)$ is optimal by providing some cases where $(X,L)$ satisfies property $N_1$ and $\mbox{rank-index}(X,L^2) >3$.

\begin{theorem}\label{thm:main theorem 3}
Suppose that ${\rm char}(\mathbb{K}) = 3$. If $X  \subset \P^r$ is a smooth complete intersection of quadrics of dimension $n \geq 3$, then the rank index of $(X,\mathcal{O}_X (2))$ equals $4$.
\end{theorem}

The proof of this result is given at the end of Section 5.

The condition $n \geq 3$ in Theorem \ref{thm:main theorem 3} ensures that the Picard group of $X$ is generated by $\mathcal{O}_X (1)$ (cf. \cite[XII, Corollary 3.7]{Gro}). In the proof of Theorem \ref{thm:main theorem 3}, one sees that this property of ${\rm Pic}(X)$ implies that $(X,\mathcal{O}_X (2))$ does not satisfy property $\QR(3)$. On the other hand, for $n=1$ and $n=2$, the group ${\rm Pic}(X)$ need not be generated by $\mathcal{O}_X (1)$ alone. For example, let $X = V (x_0 x_3 - x_1 x_2 ) \subset \P^3$ be the smooth quadric surface. Then, the Picard number of $X$ is $2$. In this case, the rank index of $(X,\mathcal{O}_X (2))$ is $3$ (cf. Example \ref{ex:quadric surface}). This shows that the assumption $n \geq 3$ in Theorem \ref{thm:main theorem 3} cannot be weakened.  \\

The organization of this paper is as follows.
Section 2 introduces the necessary preliminaries. We fix notation, recall the construction of rank 3 quadratic equations, and provide criteria for property $\QR(3)$ on Veronese varieties. These tools will be used throughout the paper.
Section 3 is devoted to the proof of Theorem \ref{thm:main theorem 1}. We establish the case $d=3$ separately and then proceed by a double induction on $n$ and $d$, re-proving several lemmas from \cite{HLMP} in characteristic 3 to complete the argument.
Section 4 contains the proofs of Theorems \ref{thm:main theorem 2} and \ref{thm:main theorem 3}. In particular, we compute the codimension of the span of rank 3 quadrics for the second Veronese embedding, and we study the second Veronese map of smooth complete intersections of quadrics in characteristic 3 to show that the bound in Theorem \ref{thm:main theorem 1}.(2) is optimal. \\

\vspace{0.2 cm}

\section{Preliminaries}
\noindent In this section, we begin by fixing notation, which we use throughout this paper, and making a few elementary remarks. Let $\mathbb{K}$ be an algebraically closed field with ${\rm char}(\mathbb{K}) \ne 2$.

\subsection{Construction of rank $3$ quadratic equations}
Let $X$ be a projective variety over $\mathbb{K}$, and let $L$ be a very ample line bundle on $X$ defining a linearly normal embedding
$$X \subset \P^r , \quad r = h^0 (X,L)-1.$$
In this subsection, we recall some basic facts about rank $3$ quadratic polynomials in the homogeneous ideal $I(X,L)$.

\begin{notation and remarks}\label{nar:producingQofR3}
Let $\varphi : H^0(X , L) \rightarrow H^0(\mathbb{P}^r, \mathcal{O}_{\mathbb{P}^r}(1))$ be the natural $\mathbb{K}$-linear isomorphism.

\renewcommand{\descriptionlabel}[1]%
             {\hspace{\labelsep}\textrm{#1}}
\begin{description}
\setlength{\labelwidth}{13mm} \setlength{\labelsep}{1.5mm}
\setlength{\itemindent}{0mm}

\item[{\rm (1)}] (Section 2 in \cite{P24}) Suppose that $L$ is decomposed as $L=L_1^{\otimes 2}\otimes L_2$ where $L_1$ and $L_2$ are line bundles on $X$ such that $p := h^0 (X,L_1 ) \geq 2$ and $q := h^0 (X,L_2 ) \geq 1$. We define the map
\begin{equation*}
Q_{L_1 , L_2} : H^0 (X,L_1 ) \times H^0 (X,L_1 ) \times H^0 (X,L_2 ) \rightarrow I(X,L)_2
\end{equation*}
by
\begin{equation*}
Q_{L_1 , L_2} (s,t,h) = \varphi (s \otimes s \otimes h) \varphi (t \otimes t \otimes h) - \varphi (s \otimes t \otimes h )^2 .
\end{equation*}
This map is well-defined, and $Q_{L_1 , L_2} (s,t,h)$ is either $0$ or else a rank $3$ quadratic equation of $X$. More precisely, $Q_{L_1 , L_2} (s,t,h)=0$ if and only if either $\{ s,t \}$ is $\mathbb{K}$-linearly dependent or else $h=0$.

\item[{\rm (2)}] (Proposition 3.1 in \cite{P24}) Suppose that $X$ is locally factorial, and let $Q \in I(X,L)_2$ be an element of rank $3$. Then there exist a decomposition
$$L=L_1^{\otimes 2}\otimes L_2$$
and an element $(s,t,h) \in H^0 (X,L_1 ) \times H^0 (X,L_1 ) \times H^0 (X,L_2 )$ such that
$$Q = Q_{L_1 , L_2} (s,t,h).$$

\item[{\rm (3)}] Let $W(L_1 , L_2 )$ denote the subspace of $I(X,L)_2$ generated by the image of $Q_{L_1 , L_2}$.

\item[{\rm (4)}] Let $\{s_1,s_2,\ldots, s_p\}$ and  $\{h_1,h_2,\ldots,h_q\}$ be any chosen bases for $H^0
(X,L_1 )$ and $H^0 (X,L_2 )$, respectively. Also, let
$$\begin{cases}
\Delta_1 := \{ (i,j)~|~ i,j \in \{1,\ldots ,p \},~i<j \},\\
\Delta_2 := \{ (i,j,k) ~|~ (i,j) \in \Delta_1 ,~k \in \{1,\ldots ,p \} - \{ i,j \} \},\\
\Delta_3 := \{ ( i,j,k, l)) ~|~ (i,j),(k, l) \in \Delta_1 ,~i<k ,~ j \neq k,l  \}, \quad \mbox{and} \\
H := \{ h_1 , \ldots , h_q \} \cup \{ h_i + h_j ~|~ 1 \le i < j \le q \}.
\end{cases}$$
We define three types of finite subsets of the image of $Q_{L_1 , L_2}$ as follows:
$$\begin{cases}
\Gamma_{11}=\big\{Q_{L_1 , L_2} (s_i,s_j,h) ~|~ (i,j) \in \Delta_1 , ~h \in H \big\}\\
\Gamma_{12}=\big\{Q_{L_1 , L_2} (s_i + s_j , s_k,h)~|~ (i,j,k) \in \Delta_2 , ~h \in H  \big\}\\
\Gamma_{22}=\big\{Q_{L_1 , L_2} (s_i+s_j,s_k +s_l,h )~|~ (i,j,k,l)  \in \Delta_3 , ~h \in H \big\}
\end{cases}$$
Let $\Gamma (L_1 , L_2)$ be the union $\Gamma_{11}\cup \Gamma_{12} \cup \Gamma_{22}$.

\item[{\rm (5)}] (Theorem 2.2 in \cite{HLMP}) The space $W(L_1 , L_2)$ is generated by $\Gamma (L_1 , L_2)$.

\item[{\rm (6)}] The map $Q_{L_1 , L_2}$ induces a finite morphism
\begin{equation*}
\widetilde{Q_{L_1 , L_2}} : \mathbb{G} \left( 1,\P H^0 (X,L_1 ) \right) \times \P H^0 (X,L_2 ) \rightarrow \P \left( I(X,L)_2 \right) 
\end{equation*}
(cf. Theorem 1.2 in \cite{P24}).
In the case of $(\mathbb{P}^n, \mathcal{O}_{\mathbb{P}^n}(d))$, we refer the reader to see Theorem 1.3 in \cite{PS25} for a few basic properties related to this morphism.

\end{description}
\end{notation and remarks}

\subsection{Criteria of Property $\QR (3)$ on Veronese varieties}
For $n \geq 1$ and $d \geq 2$, let
$$V_{n,d} := \nu_d (\mathbb{P}^n ) \subset \mathbb{P}^N ,~ N = {{n+d} \choose n}-1,$$
denote the $d$th Veronese variety. In this subsection, we fix some notation related to the Veronese varieties and provides some criteria of Property $\QR (3)$ on Veronese varieties.

\begin{notation and remarks}\label{nar:Veronese}
Define the index set
\begin{equation*}
A(n,d) := \left\{ I = (a_0, \dots, a_n) \in \mathbb{Z}_{\geq 0}^{n+1} \,\middle|\, a_0 + \cdots + a_n = d \right\},
\end{equation*}
and let
\begin{equation*}
B(n,d) := \{ z_I ~|~ I \in A(n,d) \}
\end{equation*}
be the set of standard homogeneous coordinates of th projective space $\P^N$.

\renewcommand{\descriptionlabel}[1]%
             {\hspace{\labelsep}\textrm{#1}}
\begin{description}
\setlength{\labelwidth}{13mm}
\setlength{\labelsep}{1.5mm}
\setlength{\itemindent}{0mm}

\item[$(1)$] It is well known that $I(V_{n,d})$, the homogeneous ideal of $V_{n,d}$, is generated by
\begin{equation*}
\mathcal{Q} (n,d) := \{ z_I z_J - z_K z_L ~|~ I,J,K,L \in  A(n,d)
\quad \mbox{and} \quad I+J = K+L \}.
\end{equation*}
In particular, $(\P^n , \mathcal{O}_{\P^n} (d))$ satisfies property $\QR(4)$.

\item[$(2)$] Let $W (n,d)$ be the subspace of $I(V_{n,d})_2$ generated by
\begin{equation*}
\bigcup_{1 \leq \ell \leq d/2} W (\mathcal{O}_{\P^n} (\ell) , \mathcal{O}_{\P^n} (d-2\ell) ).
\end{equation*}
Also, let $\Gamma (n,d)$ be the union of $\Gamma (\mathcal{O}_{\P^n} (\ell) , \mathcal{O}_{\P^n} (d-2\ell) )$ for $1 \leq \ell \leq d/2$. By Notation and Remarks \ref{nar:producingQofR3}.$(3)$ and $(5)$, it holds that

\begin{enumerate}
\item[$(a)$] $W (n,d)$ is the smallest subspace of $I(V_{n,d})_2$ which contains all rank $3$ quadratic forms in $I(V_{n,d})$; and
\item[$(b)$] $W (n,d)$ is generated by the finite set $\Gamma (n,d)$.
\end{enumerate}

\item[$(3)$] For two quadratic forms $Q_1 , Q_2$ in $I(V_{n,d})$, we write $Q_1 \sim Q_2$ if $Q_1 - Q_2 \in W (n,d)$, or equivalently, if $Q_1 - Q_2$ is a linear combination of rank $3$ quadratic forms in $I(V_{n,d})$. In this case, we say that $Q_1$ and $Q_2$ are \textbf{equivalent}.
\end{description}
\end{notation and remarks}

\begin{proposition}\label{prop:QR(3) criteria}
For $V_{n,d} \subset \P^N$, the following statements are equivalent:

\begin{enumerate}
\item[$(i)$] $(\P^n , \mathcal{O}_{\P^n} (d))$ satisfies property $\QR(3)$.
\item[$(ii)$] $W (n,d) = I(V_{n,d})_2$.
\item[$(iii)$] $\dim_{\mathbb{K}} ~\langle \Gamma (n,d) \rangle = \dim_{\mathbb{K}} ~ I(V_{n,d})_2$.
\item[$(iv)$] For any $I, J, K, L \in A(n,d)$ with $I + J = K + L$, it holds that $z_I z_J \sim z_K z_L$.
\end{enumerate}
\end{proposition}

\begin{proof}
The proof follows directly from Notation and Remarks \ref{nar:producingQofR3} and \ref{nar:Veronese}.
\end{proof}

\section{Proof of Theorem \ref{thm:main theorem 1}}
\noindent Throughout this section, we assume that ${\rm char}(\mathbb{K})=3$. This section aims to prove Theorem \ref{thm:main theorem 1}.$(1)$. As outlined in the Introduction, our approach to proving this theorem relies on \textbf{Step 2*} and \textbf{Step 3*}.

In Proposition \ref{boundlemma} below, we will prove \textbf{Step 2*}.

\begin{proposition}\label{boundlemma}
The rank index of $(\P^n ,\mathcal{O}_{\P^n} (3))$ is $3$ for all $n \geq 1$.
\end{proposition}

\begin{proof}
We begin by recalling Proposition 3.3 in \cite{HLMP}, which states that if $I(V_{5 ,3})$ is generated by $\Gamma (\mathcal{O}_{\P^{5}}
(1),\mathcal{O}_{\P^{5}} (1))$, then $(\P^n ,\mathcal{O}_{\P^n} (3))$ satisfies property $\QR(3)$ for all $n \geq 5$. Therefore, it suffices to prove that the rank index of $(\P^n ,\mathcal{O}_{\P^n} (3))$ is $3$ for $1 \leq n \leq 5$.

For $n=1$, we refer the reader to \cite[Corollary 2.4]{HLMP}. For $2 \leq n \leq 5$, we apply Proposition \ref{prop:QR(3) criteria}.$(iii)$. More precisely, recall that
\begin{equation*}
\Gamma (n,3) = \Gamma(\mathcal{O}_{\mathbb{P}^n}(1), \mathcal{O}_{\mathbb{P}^n}(1))
\end{equation*}
and
\begin{equation*}
\dim_{\mathbb{K}} ~ I(V_{n,3})_2 = {{n+3 \choose 3}+1 \choose 3}-{n+6 \choose 6}.
\end{equation*}
The dimension of $\langle \Gamma (n,3) \rangle$ can be computed using Macaulay2 (cf \cite{GS}). The Macaulay2 code used for this computation is available at \cite{LPS}.

Thus, by verifying the desired equality
$$\dim_{\mathbb{K}} ~\langle \Gamma (n,3) \rangle = \dim_{\mathbb{K}} ~ I(V_{n,3})_2 \quad \mbox{for} \quad 2 \leq n \leq 5,$$
we conclude that $(\P^n ,\mathcal{O}_{\P^n} (3))$ satisfies property $\QR(3)$ for all $n \geq 1$.
\end{proof}

In the remainder of this section, we shall prove the following result.

\begin{theorem}\label{thm:finite generating set of rank 3 quadrics}
$I(V_{n,d})$ is generated by $\Gamma (\mathcal{O}_{\P^n} (1),\mathcal{O}_{\P^n} (d-2))$ for all $n \geq 1$ and $d \geq 3$.
\end{theorem}

To prove Theorem \ref{thm:finite generating set of rank 3 quadrics}, we will proceed by double induction on $n \geq 1$ and $d \geq 3$. Indeed,
Theorem \ref{thm:finite generating set of rank 3 quadrics} has already been established for $n=1$ in \cite[Corollary 2.4]{HLMP} and for $d=3$ in Proposition \ref{boundlemma}. From now on, let $n \geq 2$ and $d\geq 4$. Also we assume that \\

\begin{enumerate}
\item[$(\dagger)$] the statement of Theorem \ref{thm:finite generating set of rank 3 quadrics}
is true for $(\P^{n'} , \mathcal{O}_{\P^{n'}} (d'))$ whenever $n' < n$ and $3 \leq d' < d$. \\
\end{enumerate}

\noindent Under these assumptions, we will first recall several lemmas from \cite{HLMP} that were used in the proof of Theorem \ref{thm:HLMP}.$(1)$.

\begin{notation and remarks}\label{nar:Veronese Addition}
\renewcommand{\descriptionlabel}[1]%
             {\hspace{\labelsep}\textrm{#1}}
\begin{description}
\setlength{\labelwidth}{13mm}
\setlength{\labelsep}{1.5mm}
\setlength{\itemindent}{0mm}

\item[$(1)$] Let $I,J$ be elements of $A(n,d)$. For simplicity, we denote the monomial $z_I z_J$ by $[I, J]$. For instance, an element $z_I z_J - z_K z_L$ of $\mathcal{Q} (n,d)$ is denoted by $[I,J]-[K,L]$.
\item[$(2)$]  We write an element $I \in A(n,d)$ as $I = (I_0 , I_1 , \ldots , I_n )$.
\item[$(3)$] For $I \in  A(n,d)$, we define $\Supp(I)$ as the set $\{ k ~|~ 0 \leq k \leq n, ~I_k >0 \}$.
\item[$(4)$] For each $0 \leq i \leq n$, let $e_i \in A(n,1)$ denote the element whose $i$th entry is $1$ and all others are zero.
\end{description}
\end{notation and remarks}

\begin{lemma}\label{lem:summary of lemmas}
Let $n>1$ and $d>3$, and let $I, J, K, L \in A(n,d)$ be such that $I + J = K + L$.

\renewcommand{\descriptionlabel}[1]%
             {\hspace{\labelsep}\textrm{#1}}
\begin{description}
\setlength{\labelwidth}{13mm}
\setlength{\labelsep}{1.5mm}
\setlength{\itemindent}{0mm}

\item[$(1)$] $($Lemma 4.3.$(5)$ in \cite{HLMP}$)$ Suppose that $d \geq 3$, and let $M,N \in A(n,d-2)$. Then the following equivalence holds:
\begin{equation*}
\begin{CD}
[2e_i + M, e_j + e_k + N] + [e_j + e_k + M, 2e_i + N] \\
\quad \quad \quad \quad \quad \quad \sim [e_i + e_j + M, e_i + e_k + N] + [e_i + e_k + M, e_i + e_j + N].
\end{CD}
\end{equation*}
\item[$(2)$] $($Lemma 4.4 in \cite{HLMP}$)$ Suppose that
$$\mathrm{Supp}\, I \cap \mathrm{Supp}\, J \cap \mathrm{Supp}\, K \cap \mathrm{Supp}\, L \neq \emptyset$$
or
$$\mathrm{Supp}\, I^c \cap \mathrm{Supp}\, J^c \cap \mathrm{Supp}\, K^c \cap \mathrm{Supp}\, L^c \neq \emptyset.$$
Then $[I, J] \sim [K, L]$.
\item[$(3)$] $($Lemma 4.5 in \cite{HLMP}$)$ If $I_0 \geq 3$ and $J_1, J_2 \geq 1$, then
\begin{equation*}
[I, J] \sim [-2e_0 + e_1 + e_2 + I,\ 2e_0 - e_1 - e_2 + J].
\end{equation*}
\item[$(4)$] $($Lemma 4.6 in \cite{HLMP}$)$ If $I_0, I_1, J_2, J_3 \geq 1$ and either $I_1 \geq 2$ or $J_2 \geq 2$, then
\[
[I, J] \sim [-e_0 + e_3 + I,\ e_0 - e_3 + J].
\]
\item[$(5)$] $($Lemma 4.7 in \cite{HLMP}$)$ If $I_0 = I_1 = K_0 = K_1 = 1$ and $J_0 = J_1 = L_0 = L_1 = 0$, then
$[I, J] \sim [K, L]$.
\end{description}
\end{lemma}

The proofs of Lemmas 4.3.$(5)$, 4.4, 4.5, and 4.7 in \cite{HLMP} remain valid in characteristic $3$, provided that the exceptional case $d=2$ is excluded; we therefore omit the details.

Next, we consider Lemma 4.6 in \cite{HLMP}. Indeed, this lemma also remains valid in characteristic $3$, provided that the range of $d$ is suitably adjusted. However, the argument used in the original proof does not apply in this setting; accordingly, we provide a different proof. \\

\noindent {\bf Proof of Lemma \ref{lem:summary of lemmas}.$(4)$.} When $I_1 \geq 2$, let
$$[I', J'] = [e_0+e_3+(I-e_0-e_1), e_0+e_1+(J-e_0-e_3)].$$
Then by Lemma \ref{lem:summary of lemmas}.$(1)$, it holds that
\begin{align*}
[I, J] + [I', J'] &= [e_0+e_1+(I-e_0-e_1), e_0+e_3+(J-e_0-e_3)] \\
                  & \quad +[e_0+e_3+(I-e_0-e_1), e_0+e_1+(J-e_0-e_3)] \\
                  &= [2e_0+(I-e_0-e_1), e_1+e_3+(J-e_0-e_3)] \\
                  & \quad +[e_1+e_3+(I-e_0-e_1), 2e_0+(J-e_0-e_3)] \\
                  &= [e_0-e_1+I, -e_0+e_1+J]+[-e_0+e_3+I, e_0-e_3+J] .
\end{align*}
Since $I'_1, J'_1 \geq 1$, it follows from Lemma \ref{lem:summary of lemmas}.$(2)$ that
$$[I', J'] \sim [e_0-e_1+I, -e_0+e_1+J].$$

When $J_2 \geq 2$, let
$$[I', J'] = [e_2+e_3+(I-e_0-e_3), e_0+e_3+(J-e_2-e_3)].$$
Then, by Lemma \ref{lem:summary of lemmas}.$(1)$,
\begin{align*}
[I, J] + [I', J'] &= [e_0+e_3+(I-e_0-e_3), e_2+e_3+(J-e_2-e_3)] \\
                  & \quad +[e_2+e_3+(I-e_0-e_3), e_0+e_3+(J-e_2-e_3)] \\
                  &= [2e_3+(I-e_0-e_3), e_0+e_2+(J-e_2-e_3)] \\
                  & \quad +[e_0+e_2+(I-e_0-e_3), 2e_3+(J-e_2-e_3)] \\
                  &= [-e_0+e_3+I, e_0-e_3+J]+[e_2-e_3+I, -e_2+e_3+J] .
\end{align*}
Since $I'_2, J'_2 \geq 1$, it follows from Lemma \ref{lem:summary of lemmas}.$(2)$ that $[I', J'] \sim [e_2-e_3+I, -e_2+e_3+J]$. \qed \\

\noindent \textbf{Proof of Theorem \ref{thm:finite generating set of rank 3 quadrics}.} By Proposition \ref{prop:QR(3) criteria}, we need to show that $[I,J]\sim [K,L]$ for any choice of $I,J,K,L \in A(n,d)$ with $I+J=K+L$. As mentioned after Theorem \ref{thm:finite generating set of rank 3 quadrics}, we intend to proceed with the proof by means of a double induction on $n \geq 1$ and $d \geq 3$. And for this purpose, we know that the statement of Theorem \ref{thm:finite generating set of rank 3 quadrics} holds for $n=1$ and for $d=3$, by \cite[Corollary 2.4]{HLMP} and Proposition \ref{boundlemma}, respectively. Hence, as the induction hypothesis, we assume that the statement of Theorem \ref{thm:finite generating set of rank 3 quadrics} is true for $(\P^{n'} , \mathcal{O}_{\P^{n'}} (d'))$ whenever $n' < n$ and $3 \leq d' < d$.

The original proof in \cite{HLMP} divides the argument into twelve cases and repeatedly applies \cite[Lemmas 4.4–4.7]{HLMP} to handle each case.
Figure \ref{tree} summarizes which lemma is used in each case. Except for \cite[Lemmas 4.4–4.7]{HLMP}, all other components of the proof remain valid in characteristic $3$.
Therefore, it suffices to explain how the lemmas used in the original proof are replaced in our setting. Accordingly, \cite[Lemmas 4.4- 4.7]{HLMP} are respectively replaced by Lemmas \ref{lem:summary of lemmas}.$(2)\sim (5)$ in this paper. This completes the proof. \qed  \\

\begin{figure}[H] 
  \centering \includegraphics[scale=0.5]{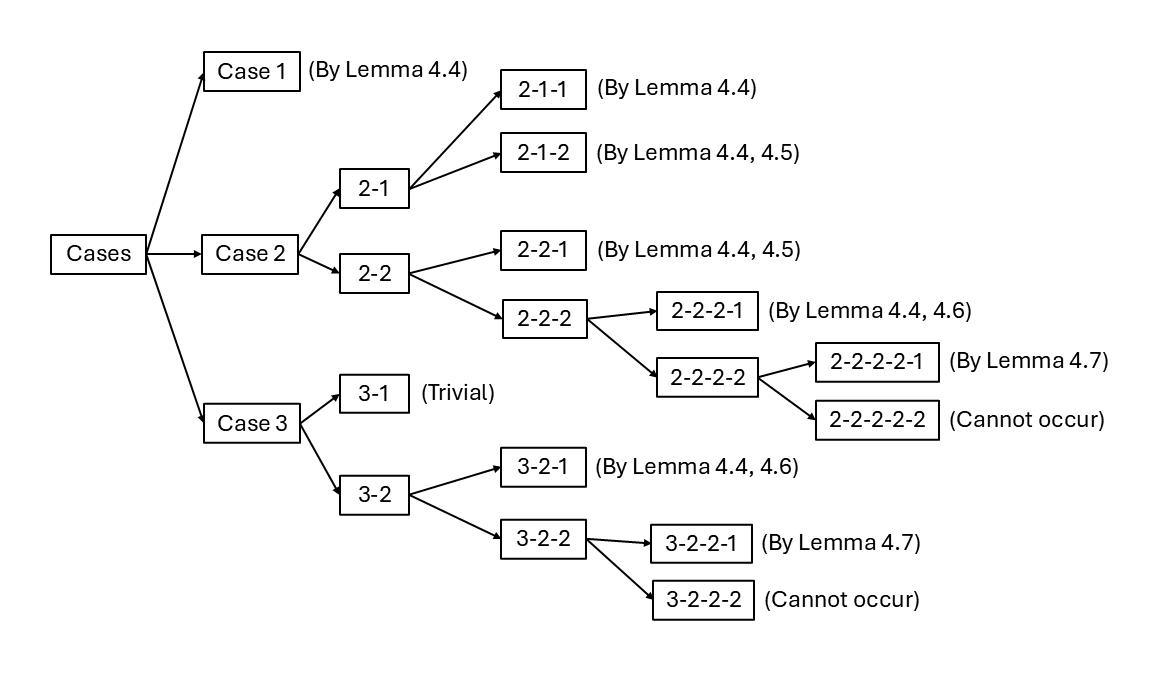} 
  \caption{Structure of the proof of Theorem 4.2 in \cite{HLMP}}
  \label{tree}
\end{figure}

\noindent {\bf Proof of Theorem \ref{thm:main theorem 1}.$(2)$.} Let $X_d \subset = \P H^0 (X,L^d )$ be the linearly normal embedding of $X$ associated with the complete linear series $|L^d|$. Let $V_{n,d} \subset \P^N$ denote the $d$th Veronese embedding of $\P^n$. As shown in the proof of Theorem 1.3 in \cite{HLMP}, the variety $X_d$ is, ideal-theoretically, a linear section of $V_{n,d}$. Since we assume $d \geq 3$, Theorem \ref{thm:main theorem 1}.$(1)$ ensures that $V_{n,d}$ satisfies property $\QR(3)$. Consequently, it follows that $(X,L^d)$ satisfies Property $\QR(3)$ as well.                \qed \\

\section{Proofs of Theorem \ref{thm:main theorem 2} and  \ref{thm:main theorem 3}}
\noindent This section is devoted to the proofs of Theorems \ref{thm:main theorem 2} and Theorem \ref{thm:main theorem 3}. Throughout this section, let $\mathbb{K}$ be an algebraically closed field of characteristic $3$ and let $n \geq 3$ be an integer.

First, let $V_{n,2} := \nu_2 (\P^n ) \subset \P^N ,~N = {n+2 \choose 2}-1$, be the second Veronese embedding of $\P^n$ for some $n \geq 3$. Let $\Phi_3 (X)$ denote the set of all rank $3$ quadrics in the projective space $\P (I(X)_2 )$. Since the Picard group of $\P^n$ is generated by $\mathcal{O}_{\P^n} (1)$, it holds that
\begin{equation*}
W(n,2) = W (\mathcal{O}_{\P^n} (1),\mathcal{O}_{\P^n} )
\end{equation*}
and
\begin{equation*}
\Gamma (n,2) = \Gamma (\mathcal{O}_{\P^n} (1),\mathcal{O}_{\P^n} ) ).
\end{equation*}
Therefore, the $\mathbb{K}$-vector space $W(n,2)$ is generated by the finite set $\Gamma (\mathcal{O}_{\P^n} (1),\mathcal{O}_{\P^n} )$ (cf. Notation and Remarks \ref{nar:Veronese}.$(2)$). Along this line, we begin with the following result.

\begin{proposition}\label{prop:lin indep}
If $W(n,2)$ denotes the $\mathbb{K}$-vector space defined above, then
$$\dim_{\mathbb{K}} ~W(n,2) =  \frac{n(n+1) (n^2 +9n+2)}{24} .$$
\end{proposition}

To verify Proposition \ref{prop:lin indep}, we introduce some notation.

\begin{notation and remarks}\label{nar:second Veronese}
Concerning the mapping $Q_{L_1 , L_2}$ in Notation and Remarks \ref{nar:producingQofR3}.$(1)$, we need to consider the decomposition 
\begin{equation*}
\mathcal{O}_{\P^n} (2) = \mathcal{O}_{\P^n} (1)^{\otimes 2} \otimes \mathcal{O}_{\P^n}.
\end{equation*}

\renewcommand{\descriptionlabel}[1]%
             {\hspace{\labelsep}\textrm{#1}}
\begin{description}
\setlength{\labelwidth}{13mm}
\setlength{\labelsep}{1.5mm}
\setlength{\itemindent}{0mm}

\item[$(1)$] Let $R = \mathbb{K} [x_0 , x_1 , \ldots , x_n ]$ be the homogeneous coordinate ring of $\P^n$. Then $\{ x_0 , x_1 , \ldots , x_n \}$ is a basis of $H^0 (\P^n, \mathcal{O}_{\P^n} (1))$.

\item[$(2)$] Rewriting the sets $\Delta_1$, $\Delta_2$ and $\Delta_3$ defined in Notation and Remarks \ref{nar:producingQofR3}.$(4)$ for $(\P^n , \mathcal{O}_{\P^n} (2))$, we have the following:
$$\begin{cases}
\Delta_1 = \{ (i,j)~|~ i,j \in \{0,1,\ldots ,n \},~i<j \},\\
\Delta_2 = \{ (i,j,k) ~|~ (i,j) \in \Delta_1 ,~k \in \{0,1,\ldots ,n \} - \{ i,j \} \}, \quad \mbox{and}\\
\Delta_3 = \{ ( i,j,k, l) ~|~ (i,j),(k, l) \in \Delta_1 ,~i<k ,~ j \neq k,l  \}.
\end{cases}$$
Thus, it holds that
\begin{align}\label{eq:Delta123}
|\Delta_1 | = {n+1 \choose 2} , ~ |\Delta_2 | = (n-1) {n+1 \choose 2} ~ \mbox{and} ~ |\Delta_3 | = 3 {n+1 \choose 4}.
\end{align}
We set
\begin{align*}
F_{i,j} &:= Q_{\mathcal{O}_{\mathbb{P}^n}(1),\mathcal{O}_{\mathbb{P}^n}}(x_i, x_j, 1) ~ \mbox{for}~(i,j) \in \Delta_1 ; \\
G_{i,j,k} &:= Q_{\mathcal{O}_{\mathbb{P}^n}(1),\mathcal{O}_{\mathbb{P}^n}}(x_i+x_j, x_k, 1)  ~ \mbox{for}~(i,j,k) \in \Delta_2  ; \\
H_{i,j,k,l} &:= Q_{\mathcal{O}_{\mathbb{P}^n}(1),\mathcal{O}_{\mathbb{P}^n}}(x_i + x_j, x_k + x_l, 1) ~ \mbox{for}~(i,j,k,l) \in \Delta_3 .
\end{align*}
Therefore, $\Gamma (\mathcal{O}_{\P^n} (1),\mathcal{O}_{\P^n}) = \Gamma_{11} \cup \Gamma_{12} \cup \Gamma_{22}$ where
$$\begin{cases}
\Gamma_{11}=\big\{ F_{i,j} ~|~ (i,j) \in \Delta_1 \big\}\\
\Gamma_{12}=\big\{ G_{i,j,k} ~|~ (i,j,k) \in \Delta_2  \big\}\\
\Gamma_{22}=\big\{H_{i,j,k,l} ~|~ (i,j,k,l)  \in \Delta_3  \big\}.
\end{cases}$$

\item[$(3)$] For every $(i,j) \in \Delta_1$, it holds that $F_{i,j}=[2e_i,2e_j]-[e_i+e_j,e_i+e_j]$.

\item[$(4)$] For distinct elements $i,j \in \{0,1,\ldots ,n \}$, we define $\overline{F_{i,j}}$ by
$$\overline{F_{i,j}}= \begin{cases}
F_{i,j} ~ & \mbox{if} ~ i>j,~ \mbox{and} \\
F_{j,i} ~ & \mbox{if} ~ i<j.
\end{cases}$$
Next, we define $G_{i,j,k}'$ for each $(i,j,k)\in\Delta_2$ and $H_{i,j,k,l}'$ for each $(i,j,k,l)\in\Delta_3$ as follows:
\begin{align*}
G_{i,j,k}'   &:= G_{i,j,k}-\overline{F_{i,k}}-\overline{F_{j,k}} \\
             & = 2[e_i+e_j,2e_k]-2[e_i+e_k,e_j+e_k] \\
H_{i,j,k,l}' &:= H_{i,j,k,l}-G_{i,j,k}-G_{i,j,l}-G_{k,l,i}-G_{k,l,j}+F_{i,k}+F_{i,l}+\overline{F_{j,k}}+\overline{F_{j,l}} \\
             & = 4[e_i + e_j, e_k + e_l] - 2[e_i + e_k, e_j + e_l]- 2[e_i + e_l, e_j + e_k]
\end{align*}
Finally, we define $\Gamma_{12}'$ and $\Gamma_{22}'$ as
\begin{equation*}
\Gamma_{12}':=\{G_{i,j,k}' ~|~ (i,j,k)\in\Delta_2\} \quad \mbox{and} \quad  \Gamma_{22}':=\{H_{i,j,k,l}' ~|~(i,j,k,l)\in\Delta_3\}.
\end{equation*}
\end{description}
\end{notation and remarks}
Before proving Proposition \ref{prop:lin indep}, we establish the following lemma.
\begin{lemma}\label{lemma:G'}
$|{\Gamma_{22}}'| = \frac{1}{3} \times |\Delta_3 |$.
\end{lemma}

\begin{proof}
From the definition of $\Gamma_{22}'$, we have a surjective map
$$\phi:\Delta_3\rightarrow \Gamma_{22}'$$
defined by $\phi((i,j,k,l))=H'_{i,j,k,l}$. It suffices to prove that $\phi$ is a $3$-to-$1$ map.

Since ${\rm char}(\mathbb{K}) = 3$, we have
\begin{equation*}
H'_{i,j,k,l}=[e_i + e_j, e_k + e_l]+[e_i + e_k, e_j + e_l]+[e_i + e_l, e_j + e_k]
\end{equation*}
(cf. Notation and Remarks \ref{nar:second Veronese}.$(4)$), and so
\begin{equation}\label{eq:H'}
H'_{i,j,k,l}= \begin{cases}
H'_{i,k,j,l}=H'_{i,l,j,k} ~ & \mbox{if} ~ j <k ;   \\
H'_{i,k,j,l}=H'_{i,l,k,j} ~ & \mbox{if} ~ k<j<l ; \\
H'_{i,k,l,j}=H'_{i,l,k,j} ~ & \mbox{if} ~ l<j .\\
\end{cases}
\end{equation}
This shows that $\phi^{-1} (\tau)$ has at least three elements for every $\tau \in \Delta_3$.

Next, we will show that if $(i,j,k,l)$ and $(i',j',k',l')$ are two elements in $\Delta_3$ such that
$H_{i,j,k,l}' = H_{i',j',k',l'}'$, then $(i',j',k',l')$ is obtained by permuting the entries of $(i,j,k,l)$. To this aim, we use the lexicographic monomial order on the homogeneous coordinate ring of $\P^N$ such that the variable order is defined as follows: For $0 \leq i ,j,k,l \leq n$ with $i \leq j$ and $k \leq l$, we define
\begin{equation*}
z_{e_i+e_j} > z_{e_k+e_l} ~ \mbox{if} ~ \begin{cases} i=j, ~k=l ~ \mbox{and} ~i<k ;~ \mbox{or} \\
    i=j ~\mbox{and}~k<l; ~ or\\
                                                      i<j, ~ k<l, ~ \mbox{and either both}~ i=k ~\mbox{and}~ j<l, ~ \mbox{or}~ i<k.
                                                      \end{cases}
\end{equation*}
Indeed, we may interpret each $z_{i, j}$ with $0\leq i \leq j \leq n$ as a lattice point $(i, j)$ in the upper triangular region including the main diagonal.
Accordingly, we define the order on the variables $z_{e_i+e_j}$ as follows.
Among the diagonal points $i=j$, those with smaller indices have higher order.
When comparing a diagonal point $(i,i)$ with a point $(k,l)$ not on the diagonal, the diagonal point always has higher order.
Finally, for two points $(i,j)$ and $(k,l)$ not on the diagonal with $i<j$ and $k<l$, the order is determined by the indices: $z_{e_i+e_j}$ has higher order if $i<k$, or if $i=k$ and $j<l$.

For $1 \leq p \leq 4$, we denote by $\sigma_p(i,j,k,l)$ the $p$th smallest number among $i,j,k,l$. Thus $\sigma_1(i,j,k,l) = i$ if $(i,j,k,l) \in \Delta_3$. One can check that the leading monomial $\mathrm{LM}(H'_{i,j,k,l})$ of $H'_{i,j,k,l}$ is given as
\begin{equation}\label{eq:leading term}
\mathrm{LM}(H'_{i,j,k,l})   = z_{e_i+e_{\sigma_2(i,j,k,l)}} z_{e_{\sigma_3(i,j,k,l)}+e_{\sigma_4(i,j,k,l)}} ~ \mbox{for every} ~ (i,j,k,l) \in \Delta_3 .
\end{equation}
Now, let $(i,j,k,l)$ and $(i',j',k',l')$ be two elements of $\Delta_3$ such that $H_{i,j,k,l}' = H_{i',j',k',l'}'$. Then
$$\mathrm{LM}(H'_{i,j,k,l}) = \mathrm{LM}( H_{i',j',k',l'}' )$$
and hence
$$\sigma_p (i,j,k,l)=\sigma_p (i',j',k',l') ~ \mbox{for all}~ 1 \leq p \leq 4$$
by (\ref{eq:leading term}). This implies that $(i',j',k',l')$ is obtained by permuting the entries of $(i,j,k,l)$.
Since $(i',j',k',l') \in \Delta_3$, it must be one of the following three possibilities:
\begin{align*}
&(i,\sigma_2(i,j,k,l),\sigma_3(i,j,k,l),\sigma_4(i,j,k,l)),\\
&(i,\sigma_3(i,j,k,l),\sigma_2(i,j,k,l),\sigma_4(i,j,k,l)),\\
&(i,\sigma_4(i,j,k,l),\sigma_2(i,j,k,l),\sigma_3(i,j,k,l)).
\end{align*}
This completes the proof that $\phi: \Delta_3 \rightarrow \Gamma_{22}'$ is a $3$-to-$1$ map.
\end{proof}

Lemma \ref{lemma:G'} now allows us to prove Proposition \ref{prop:lin indep}. \\

\noindent {\bf Proof of Proposition \ref{prop:lin indep}.} Note that, by Notation and Remarks \ref{nar:second Veronese}.$(3)$, the $\mathbb{K}$-space $W(n,d)$ is generated by the union
$$\Gamma ' := \Gamma_{11} \cup \Gamma'_{12}\cup \Gamma'_{22}.$$
Secondly, one can see that $\Gamma'$ is $\mathbb{K}$-linearly independent. Indeed, with the monomial ordering defined in the proof of Lemma \ref{lemma:G'}, we have
\begin{alignat*}{3}
&\mathrm{LM}(F_{i,j})        &&= z_{2e_i}z_{2e_j}    && \text{for}~(i,j) \in \Delta_1 ,\\
&\mathrm{LM}(G'_{i,j,k})     &&= z_{e_i+e_j}z_{2e_k} && \text{for}~(i,j,k)\in\Delta_2 , ~\mbox{and} \\
&\mathrm{LM}(H'_{i,j,k,l})   &&= z_{e_i+e_{\sigma_2(i,j,k,l)}} z_{e_{\sigma_3(i,j,k,l)}+e_{\sigma_4(i,j,k,l)}} &&\text{for}(i,j,k,l) \in \Delta_3 .
\end{alignat*}
This shows that the leading monomials of elements in $\Gamma'$ are all distinct, which apparently implies the $\mathbb{K}$-linear independence of $\Gamma'$.

Now, by (\ref{eq:Delta123}) and Lemma \ref{lemma:G'}, we have
\begin{align*}
\dim_{\mathbb{K}} ~W(n,d) && = & ~ |\Gamma ' |        \\
                          && = & ~ |\Gamma_{11}| + |\Gamma'_{12}| + |\Gamma'_{22}| \\
                          && = & ~ |\Delta_1| + |\Delta_2| + \frac{1}{3} |\Delta_3| \\
                          && = & ~ {n+1 \choose 2} + (n-1) {n+1 \choose 2} + \frac{1}{3} \times 3 \times \binom{n+1}{4} \\
                          && = & ~ \frac{n^2 +9n+2}{12} \times {n+1 \choose 2}.
\end{align*}
This completes the proof. \qed \\

Now, we can see that Theorem \ref{thm:main theorem 2} follows immediately from Proposition \ref{prop:lin indep}.
\\

\noindent {\bf Proof of Theorem \ref{thm:main theorem 2}.} The assertion follows immediately from Proposition \ref{prop:lin indep}. To be precise, it holds that
\begin{align*}
\mbox{codim} (  \langle \Phi_3 (X) \rangle ,\P ( I(X)_2 )) && = & ~ \dim_{\mathbb{K}} ~ I(X)_2 -  \dim_{\mathbb{K}}~W(n,d)         \\
                          && = & ~ {{n+2 \choose 2} +1 \choose 2} -  \frac{n(n+1) (n^2 +9n+2)}{24}   \\
                          && = & ~  {n+1 \choose 4}.
\end{align*}
This completes the proof. \qed \\

Next, we prove Theorem \ref{thm:main theorem 3}. Let $X  \subset \P^r$ be an $n$-dimensional smooth complete intersection of quadrics with $n \geq 3$. By the Lefschetz hyperplane theorem \cite[XII, Corollary 3.7]{Gro}, the Picard group of $X$ is generated by $\mathcal{O}_X (1)$. \\

\noindent {\bf Proof of Theorem \ref{thm:main theorem 3}.} Again, we denote by
$$V_{r,2} := \nu_2 (\P^r ) \subset \P^N , ~N= {r+2 \choose 2}-1,$$
the second Veronese embedding of $\P^r$. Then
$$\langle \nu_2 (X) \rangle = \P^{N-r+n}$$
and the natural map
$$g : I(V_{r,2} / \P^N )_2 \rightarrow I(\nu_2 (X) / \P^{N-r+n} )_2$$
is an isomorphism since $V_{r,2}$ is arithmetically Cohen-Macaulay. Now, consider the following commutative diagram
\begin{equation*}
\begin{CD}
H^0 (\P^r , \mathcal{O}_{\P^r} (1)) \times H^0 (\P^r , \mathcal{O}_{\P^r} (1)) & \quad \stackrel{\alpha}{\rightarrow} \quad & I(V_{r,2} / \P^N )_2 \\
\downarrow f & & \downarrow g \\
H^0 (X , \mathcal{O}_{X} (1)) \times H^0 (X , \mathcal{O}_{X} (1)) & \quad \stackrel{\beta}{\rightarrow}              \quad & I(\nu_2 (X) / \P^{N-r+n} )_2
\end{CD}
\end{equation*}
where $\alpha = Q_{\mathcal{O}_{\P^r} (1) , \mathcal{O}_{\P^r}}$ and $\beta = Q_{\mathcal{O}_{X} (1) , \mathcal{O}_{X}}$ (cf. Notation and Remarks \ref{nar:producingQofR3}.$(1)$). Every rank $3$ quadratic polynomial in the ideal of $\nu_2 (X)$ arises as the image $\alpha$, since $X$ is smooth and its Picard group is generated by $\mathcal{O}_X (1)$ (cf. Notation and Remarks \ref{nar:producingQofR3}.$(2)$). Theorem \ref{thm:main theorem 3} shows that the image of $\alpha$ spans a proper subspace of $ I(V_{r,2} / \P^N )_2$ since $r \geq 4$. Consequently, the image of $\beta$ is a proper subspace of $I(\nu_2 (X) / \P^{N-r+n} )_2$. This completes the proof that $(X,\mathcal{O}_X (2))$ fails to satisfy property $\QR (3)$. On the other hand, the reason why $(X,\mathcal{O}_X (2))$ satisfies property $\QR (4)$ is straightforward: indeed, $\nu_2 (X)$ is, ideal-theoretically, a linear section of $V_{r,2}$, and $(\P^r,\mathcal{O}_{\P^r} (2))$ satisfies property $\QR (4)$. Therefore, the rank index of $(X,\mathcal{O}_X (2))$ is exactly $4$.           \qed \\

Finally, we present an example showing that the hypothesis $n \geq 3$ in Theorem \ref{thm:main theorem 3} cannot be weakened.

\begin{example}\label{ex:quadric surface}
Let $X = V(x_0x_3-x_1x_2) \subset \mathbb{P}^3$ be the smooth quadric surface. Then $X = \P^1 \times \P^1$, so that
\begin{equation*}
{\rm Pic}(X) \cong \Z [\mathcal{O} (1,0)] \oplus \Z [\mathcal{O} (0,1)],
\end{equation*}
and
\begin{equation*}
X = \{[su : sv : tu : tv ] \mid [s, t],[u, v] \in \mathbb{P}^1 \}.
\end{equation*}
Now, consider the surface $S := \nu_2 (X) \subset \P^8$. Then
\begin{equation*}
S = \{[s^2 u^2 : s^2 uv : st u^2 : stuv :  s^2 v^2 : stv^2 : t^2 u^2 : t^2 uv : t^2 v^2] \mid [s, t],[u, v] \in \mathbb{P}^1 \},
\end{equation*}
so that $S$ can be regarded as the hyperplane section of
\begin{equation*}
V_{3,2} = \{ [x_0^2 :x_0 x_1 :x_0 x_2 :x_0 x_3: x_1^2 :x_1x_2 :x_1 x_3 : x_2^2 : x_2x_3 :x_3^2 ] ~|~[x_0,x_1,x_2] \in \P^2 \} \subset \P^9
\end{equation*}
by the hyperplane $V (z_3 - z_5 )$, where $z_0 , z_1 , \ldots , z_9$ are the homogeneous coordinates of $\P^9$.

We will show that $(X,\mathcal{O}_X (2))$ satisfies property $\QR(3)$. Indeed, by Theorems \ref{thm:main theorem 2} and the proof of Theorems \ref{thm:main theorem 3}, the set $\Gamma (\mathcal{O}_X (1),\mathcal{O}_X )$ spans a $19$-dimensional subspace of the $20$-dimensional space $I(S)_2$. From the decomposition
\begin{equation*}
\mathcal{O}_X (2) =  \mathcal{O} (1,0)^2 \otimes \mathcal{O} (0,2) ,
\end{equation*}
we obtain the following rank $3$ quadratic polynomials in $I(S)_2$:
\begin{align*}
Q_1 &= \varphi(s^2 u^2)\varphi(t^2 u^2) - \varphi(stu^2)^2
= \varphi(x_0^2)\varphi(x_2^2) - \varphi(x_0 x_2)^2
= z_0 z_9 - z_2^2, \\
Q_2 &= \varphi(s^2 uv)\varphi(t^2 uv) - \varphi(stuv)^2
= \varphi(x_1 x_2)\varphi(x_3 x_2) - \varphi(x_1 x_3)^2
= z_1 z_8 - z_3^2, \\
Q_3 &= \varphi(s^2 v^2)\varphi(t^2 v^2) - \varphi(stv^2)^2
= \varphi(x_1^2)\varphi(x_3^2) - \varphi(x_1 x_3)^2
= z_4 z_9 - z_6^2.
\end{align*}
Note that $Q_1$ and $Q_2$ are contained in $\Gamma (\mathcal{O}_X (1),\mathcal{O}_X )$, whereas $Q_2$ is not. It follows that the union $\Gamma (\mathcal{O}_X (1),\mathcal{O}_X ) \cup \{ Q_2 \}$ spans $I(S)_2$, which proves that $S \subset \P^8$ satisfies property $\QR (3)$.
In connection with this example, we also refer the reader to Theorem 6.2 in \cite{MP}.
\end{example}

\end{document}